\documentclass[reqno]{amsart}
\usepackage{graphicx}
\usepackage{geometry}
\usepackage[leqno]{amsmath}
\usepackage{amscd}
\usepackage{amsthm}
\usepackage{amsfonts}
\usepackage{amssymb}
\usepackage[pdfencoding=auto,colorlinks=true]{hyperref}
\usepackage{xcolor}
\usepackage{mathrsfs}
\newtheorem{theorem}{Theorem}[section]

\newtheorem{corollary}[theorem]{Corollary}
\newtheorem{proposition}[theorem]{Proposition}
\newtheorem{lemma}[theorem]{Lemma}

\theoremstyle{definition}

\newtheorem{example}[theorem]{Example}

\newtheorem{remark}[theorem]{Remark}
\makeatletter
\newcommand{\leqnomode}{\tagsleft@true}
\newcommand{\reqnomode}{\tagsleft@false}
\makeatother

\DeclareMathOperator{\Aut}{Aut}

\DeclareMathOperator{\Cent}{Cent}
\DeclareMathOperator{\Perm}{Perm}

\newcommand{\gen}[1]{\langle #1 \rangle} 

\newcommand{\B}{\mathfrak{B}}

\newcommand{\id}{\mathrm{id}}

\numberwithin{equation}{section}

\begin{document}
\newtheorem*{thm}{Theorem}

\title{Abelian maps, bi-skew braces, and opposite pairs of {H}opf-{G}alois structures}
\author{Alan Koch}
\address{Department of Mathematics, Agnes Scott College, 141 E. College Ave., Decatur, GA\ 30030 USA \\akoch@agnesscott.edu}
\date{\today       }

\begin{abstract} Let $G$ be a finite nonabelian group, and let $\psi:G\to G$ be a homomorphism with abelian image. We show how $\psi$ gives rise to two Hopf-Galois structures on a Galois extension $L/K$ with Galois group (isomorphic to) $G$; one of these structures generalizes the construction given by a ``fixed point free abelian endomorphism'' introduced by Childs in 2013. We construct the skew left brace corresponding to each of the two Hopf-Galois structures above. We will show that one of the skew left braces is in fact a bi-skew brace, allowing us to obtain four set-theoretic solutions to the Yang-Baxter equation as well as a pair of Hopf-Galois structures on a (potentially) different finite Galois extension.  
	\end{abstract}

\maketitle

\section{Introduction}

Let $G$ be a finite nonabelian group, and let $L/K$ be a Galois extension with Galois group $G$. In \cite{Childs13} Childs introduces the notion of a {\it fixed point free abelian endomorphism} of $G$. Given such a map $\psi:G\to G$ one can endow $L/K$ with a Hopf-Galois structure. Childs furthermore provides a criterion to determine when two different choices of fixed point free abelian endomorphism yield the same Hopf-Galois structure. 

In this work, we introduce a generalization to the above theory. By adjusting how an endomorphism gives rise to a Hopf-Galois structure we are able to improve upon the results in \cite{Childs13} in several meaningful ways. First, we are able to drop the fairly restrictive ``fixed point free'' condition, thereby obtaining a larger family of Hopf-Galois structures. Indeed, under Childs's classification, each Hopf-Galois structure emanates from the same Hopf algebra; our generalization allows for more Hopf algebras to act on $L/K$. Second, we simplify the criterion to determine whther endomorphisms give the same Hopf-Galois structure. Finally, in most circumstances (including every case which arises from a fixed point free abelian endomorphism) we are able to find a second Hopf-Galois structure on $L/K$ which is related to the first. While this second structure has been well-known since 1987 (see \cite{GreitherPareigis87}), Childs's theory lacks an explicit way to describe it; using our modified correspondence this structure becomes transparent.

The past five years have seen a resurgence in Hopf-Galois theory on Galois extensions due to their relationship with set-theoretic solutions to the Yang-Baxter equation. Guarnieri and Vendramin \cite{GuarnieriVendramin17} introduced the notion of skew left braces to provide non-degenerate set-theoretic solutions to this equation, building on the work of Rump, who in \cite{Rump07} developed braces to find solutions that are also involutive. In \cite{Bachiller16} a connection is given between regular, $G$-stable subgroups of $\Perm(G)$ and skew left braces, making the bridge from Hopf-Galois structures and solutions to the Yang-Baxter equation complete: any Hopf-Galois structure on a Galois extension gives rise to such a solution. 

Here, given an abelian map $\psi$ we construct what Childs in \cite{Childs19} calls a {\it bi-skew brace}. A bi-skew brace is a skew left brace which remains a skew left brace upon interchanging the two binary operations. That the skew left brace is bi-skew has several consequences. First, in the theory of Hopf-Galois extensions a bi-skew brace gives rise to two more Hopf-Galois structures, generally on a Galois extension with a different Galois group (that is, one not isomorphic to $G$). Second, a bi-skew brace gives more solutions to the Yang-Baxter equation; using the theory of skew brace opposites as developed independently by Rump in \cite{Rump19} and the author with Truman in \cite{KochTruman20}, a single abelian endomorphism can give up to four different set-theoretic solutions to the Yang-Baxter equation. Two of these solutions appear in \cite{KochStordyTruman20} in the fixed point free case. 

After a quick survey of the basic background material, we introduce abelian maps. Theorem \ref{main} establishes the regular, $G$-stable subgroup associated to an abelian map, while corollary \ref{main2} gives the opposite structure. The connection to \cite{Childs13} is given, along with a secondary link to cases where $G$ decomposes as an internal semidirect product as in \cite{CrespoRioVela16}. The bi-skew brace corresponding to an abelian map is also given, along with (up to) four solutions to the Yang-Baxter equation. We will also find all abelian maps on the symmetric groups $S_n,\;n\ge 5$, the metacyclic groups $M_{p,q}$ of order $pq$ with $p,q$ prime, and all dihedral groups $D_n$.

Throughout, $G$ is a finite, nonabelian group with center $Z(G)$, and $L/K$ is a Galois extension with Galois group $G$. We denote by $C_n$ the cyclic group of order $n$, written multiplicatively.

\section{Background}

Here, we will provide much of the background needed for the subsequent sections.

\subsection{Hopf Galois extensions and Greither-Pareigis theory} 
In \cite{GreitherPareigis87}, Greither and Pareigis develop a powerful theory to find Hopf-Galois structures which we shall briefly outline here--see, e.g., \cite{Childs00} for a detailed treatment. 

Let $\Perm(G)$ denote the group of permutations of $G$. For $\eta\in\Perm(G)$ we will denote the image of $g\in G$ under $\eta$ by $\eta[g]$. We say a subgroup $N\le\Perm(G)$ is {\it regular} if for all $g,h\in G$ there is a unique $\eta\in N$ such that $\eta[g]=h$. 
The simplest examples of such subgroups are the image of $G$ under left regular representation $\lambda:G\to \Perm(G)$ and right regular representation $\rho:G\to \Perm(G)$. Notice that these two regular subgroups commute with each other: $\lambda(g)\rho(h)=\rho(h)\lambda(g)$ for all $g,h\in G$.

Clearly, $\lambda(G)$ acts on $\Perm(G)$ via conjugation, i.e. $(\lambda(g),\eta)\mapsto \;^g\eta:=\lambda(g)\eta\lambda(g^{-1})\in \Perm(G)$ for $g\in G,\; \eta\in \Perm(G)$. A subgroup $N\le \Perm(G)$ is said to be {\it $G$-stable} if $^g\eta\in N$ for all $g\in G,\; \eta\in N$. Note that since $^g\lambda(h)=\lambda(ghg^{-1})\in\lambda(G)$ and $^g\rho(h)=\rho(g)\in\rho(G)$ both $\lambda(G)$ and $\rho(G)$ are $G$-stable.

Suppose $N\le\Perm(G)$ is regular and $G$-stable. By \cite[Lemma 2.4.2]{GreitherPareigis87} the subgroup $N'=\Cent_{\Perm(G)}(N)=\{\eta'\in N': \eta\eta'=\eta'\eta \text{ for all }\eta\in N\}$ is regular, $G$-stable, and is isomorphic to $N$. We will call this the {\it opposite subgroup} to $N$, terminology which is justified by the fact that there is a canonical isomorphism $N'\to N^{\text{opp}}$ \cite[Lemma 2.4.2]{GreitherPareigis87}. As a simple example, $\lambda(G)'=\rho(G)$.

In \cite{GreitherPareigis87} a one-to-one correspondence between Hopf-Galois structures on $L/K$ and regular, $G$-stable subgroups is given.  For $N\le \Perm(G)$ the corresponding $K$-Hopf algebra is the fixed ring $L[N]^G$, where $g\in G$ acts on $N$ via conjugation via $\lambda(g)$ and on $L$ through the Galois action. Furthermore, we will call the Hopf-Galois structure obtained from $N'$ the {\it opposite Hopf-Galois structure} to $L[N]^G$.  

Taking $N=\rho(G)$ gives the usual Galois action. Taking $N=\lambda(G)$ produces what is called the {\it canonical nonclassical} Hopf Galois structure in  \cite{Truman16}. This structure will be of particular importance here and we shall denote its Hopf algebra by $H_{\lambda}$.

Generally, for $N\le\Perm(G)$ regular, $G$-stable, the corresponding Hopf-Galois structure with Hopf algebra  $H=L[N]^G$ is said to be of {\it type $N$}.
Note that $\lambda(G)$ and $\rho(G)$ are evidently of type $G$.

\subsection{Fixed point free abelian endomorphisms}

Childs's use of fixed point free abelian endomorphism in \cite{Childs13} provides a useful construction of regular, $G$-stable subgroups. Here, an endomorphism $\psi:G\to G$ is said to be {\it fixed point free} if for every nontrivial $g\in G$ we have $\psi(g)\ne g$; and
{\it abelian} if $\psi(G)\le G$ is abelian. For brevity, we will often refer to our endomorphisms as ``maps''. Note that an abelian map is constant on conjugacy classes.
A classification of fixed point free abelian maps on certain classes of finite groups is well understood: see \cite{Childs13,Caranti13,KochStordyTruman20}.

Given a fixed point free abelian endomorphism $\psi:G \to G$, we let
\[N=N_{\psi}=\{\lambda(g)\rho(\psi(g)):g\in G\}.\]
It is easy to verify that $N$ is a regular, $G$-stable subgroup of $\Perm(G)$. Furthermore, $N\cong G$ via the map $\lambda(g)\rho(\psi(g))\mapsto g$.

By \cite[Th. 2]{Childs13} we know that $N_{\psi_1}=N_{\psi_2}$ if and only if there is a fixed point free homomorphism $\zeta:G\to Z(G)$ (necessarily abelian, of course) such that 
$\psi_2(g) = \psi(g\zeta(g^{-1})\zeta(g)$
for all $g\in G$. 

As the $N$ constructed above is regular and $G$-stable, $\psi:G\to G$ gives a Hopf-Galois structure on $L/K$. As mentioned in \cite{Childs13} and explored in greater detail in \cite{KochKohlTrumanUnderwood19c}, we have $(L[N])^G\cong H_{\lambda}$ as $K$-Hopf algebras, however the precise action on $L$ will be different from the canonical nonclassical action unless $\psi(G)\le Z(G)$.

\subsection{Braces and the Yang-Baxter equation}
Regular, $G$-stable subgroups allow us to construct set-theoretic solutions to the Yang-Baxter equations via skew left braces. 

At present time, there is not a standard notation for skew left braces. We will mostly follow the notation in \cite{GuarnieriVendramin17}, writing $\B=(B,\cdot,\circ)$ for the skew left brace, where $(B,\cdot)$ and $(B,\circ)$ are groups and, for all $x,y,z\in B$, 
$x\circ(y\cdot z) =(x\circ y)\cdot x^{-1} \cdot (x\circ z)$.  We will write $xy$ for $x\cdot y$ and $\overline x$ for the inverse to $x$ under $\circ$. We will refer to the operations as the dot and circle operations; some works call these the additive and multiplicative operations for historical reasons. We will denote the identity, common to both group structures, by $1_B$. 

Skew left braces were introduced by Guarnieri and Vendramin \cite{GuarnieriVendramin17}, generalizing the notion of {\it left brace} formulated by Rump \cite{Rump07} who required that $(B,\cdot)$ be abelian. For simplicity, we will use ``brace'' to mean ``skew left brace''. We will only consider braces with $B$ finite.

Two simple examples can be found using the group $G$ as the underlying set. We can let $\B=(G,\cdot,\circ)$ where both $(G,\cdot)$ and $(G,\circ)$ are the usual group operation on $G$ (i.e., $g\cdot h =g\circ h = gh$ for all $g,h\in G$): we call this the {\it trivial brace} on $G$. Alternatively, we can let $\B=(G,\cdot,\circ)$ with $(G,\cdot)$ the usual group operation and $g\circ h = hg$: we call this the {\it almost trivial brace} on $G$. 

Bachiller, in \cite{Bachiller16}, describes a connection between braces and regular, $G$-stable subgroups, hence to Hopf-Galois structures, as follows. Suppose $(N,\cdot)\le \Perm(G)$ is regular and $G$-stable. Let $\varkappa:N \to G$ be the map $\varkappa(\eta) = \eta[1_G]$: as $N$ is regular, $\varkappa$ is a bijection. Then $(N,\cdot,\circ)$ is a brace with 
$\eta\circ \pi = \varkappa^{-1}(\varkappa(\eta)\ast_G\varkappa(\pi)), \;\eta,\pi\in N$
where $\ast_G$ is the usual operation on $G$. The brace constructed not only incorporates $N$ (as $(N,\cdot)$) but also $G$ since $\varkappa: (N,\circ)\to G$ is an isomorphism.
It is easy to see that $\lambda(G)\le \Perm(G)$ gives the trivial brace on $G$, and $\rho(G)\le \Perm(G)$ gives the almost trivial brace on $G$.


If $\B=(B,\cdot,\circ)$ is a brace and $(B,\cdot)$ is a nonabelian group, the notion of an opposite brace was developed independently in \cite{Rump19} and \cite{KochTruman20}. The opposite brace is defined as $\B'=(B,\cdot',\circ)$ where $x\cdot' y = yx$. While the underlying groups of $\B$ and $\B'$ are isomorphic (i.e., $(B,\cdot)\cong (B,\cdot')$ and, of course, $(B,\circ)\cong(B,\circ)$), in general $\B\not\cong \B'$. Certainly, $\B=\B'$ if and only if $(B,\cdot)$ is abelian, however there exist examples of braces isomorphic to their opposite with $(B,\cdot)$ nonabelian: see \cite[Ex. 6.1]{KochTruman20}.
Evidently, the trivial brace and almost trivial brace are opposites.


Braces were developed to find certain set-theoretic solutions to the Yang-Baxter equation. A {\it set-theoretic solution to the Yang-Baxter equation} is a set $B$ together with a function $R:B\times B\to B\times B$ such that
	\[(R\times\id)(\id \times R)(R\times \id) = (\id\times R)(R\times \id)(\id\times R):B\times B\times B\to B\times B\times B.\]
	Writing $R(x,y)=(R_x(y),R_y(x))$, then $R$ is {\it non-degenerate} if both $R_x$ and $R_y$ are bijections. Also, if $R(R(x,y))=(x,y)$ then $R$ is {\it involutive}. 

%
For any brace $\B=(B,\cdot,\circ)$ we let
\[R_{\B}(x,y) = (x^{-1}(x\circ y),\overline{x^{-1}(x\circ y)}\circ x \circ y)\]
for $x,y\in B$. Then $R_{\B}$ is a non-degenerate solution to the Yang-Baxter equation. Furthermore, $R_{\B}$ is involutive if and only if $(B,\cdot)$ is abelian. 

If $(B,\cdot)$ is nonabelian, then the opposite brace gives an additional solution
\[R_{\B'}(x,y)=((x\circ y)x^{-1},\overline{(x\circ y)x^{-1}}\circ x \circ y), x,y\in B;\]
furthermore $R_{\B'}$ is the inverse to $R_{\B}$ (c.f, e.g., \cite[Th. 4.1]{KochTruman20}).

For example, applying these constructions to the trivial brace gives
\[R_{\B}(x,y)=(y,y^{-1}xy),\;R_{\B'}(x,y)=(x,x^{-1}yx),\;x,y\in B.\]

\section{Abelian Maps and Regular Subgroups}

We now show how to construct a regular, $G$-stable subgroup from an abelian map. Notice that we cannot simply drop the ``fixed point free'' condition and use Childs's construction. This is easy to see: if $\psi:G \to G$ is an abelian map and $\psi(g)=g,\;g\ne 1_G$ then $\lambda(g)\rho(\psi(g))[1_G]= \lambda(1_G)\rho(\psi(1_G))[1_G]$
and so $\{\lambda(g)\rho(\psi(g)):g\in G\}$ is not regular.

Our main result is as follows.

\begin{theorem}\label{main}
	Let $\psi:G\to G$ be an abelian endomorphism. For $g\in G$, define $\eta_g\in\Perm(G)$ by $\eta_g[h]=g\psi(g^{-1})h\psi(g),\;h\in G$,
	and let $N=N_{\psi}=\{\eta_g:g\in G\}$. Then $N$ is a regular, $G$-stable subgroup of $\Perm(G)$.
\end{theorem}

In section \ref{Guru} we will show how to obtain Childs's construction from ours.

\begin{proof}
	We will first show $N$ is in fact a subgroup of $\Perm(G)$. Note that, for $g,h,k\in G$, 
	\begin{align*}\eta_g\eta_{k^{-1}}[h]&=\eta_g[k^{-1}\psi(k)h\psi(k^{-1})]\\
	&=g\psi(g^{-1})k^{-1}\psi(k)h\psi(k^{-1})\psi(g)\\
	&=\big(g\psi(g^{-1})k^{-1}\psi(g)\big)\psi(g^{-1}) \psi(k) h\psi(k^{-1}g).\end{align*}
	Now since $\psi$ is abelian we have $\psi(g^{-1}k) =\psi\big(g\psi(g^{-1})k^{-1}\psi(g)\big)^{-1}$, hence
\[
	\eta_g\eta_{k^{-1}}[h]=\big(g\psi(g^{-1})k^{-1}\psi(g)\big)\psi\big(g\psi(g^{-1})k^{-1}\psi(g)\big)^{-1}h\big(g\psi(g^{-1})k^{-1}\psi(g)\big)
=\eta_{g\psi(g^{-1})k^{-1}\psi(g)}[h].\]
	As $\eta_g\eta_{k^{-1}}\in N$ we get that $N\le \Perm(G)$.
	
	We next show $N$ is regular. Suppose $\eta_g[h]=h$ for some $g,h\in G$. Then $g\psi(g^{-1})h\psi(g)=h$. Thus,
$\psi(g)\psi(h)=\psi(gh)=\psi(g\psi(g^{-1})h\psi(g))=\psi(h)$
	and so $g\in \ker \psi$. Therefore,
	$\eta_g[h] = g\psi(g^{-1})h\psi(g) = gh = h$
	and $g=1_G$. Furthermore, since $\eta_g[1_G]=g$ we see that the $\eta_g$ are all distinct, hence $|N|=|G|$. These properties suffice to show that $N\le \Perm(G)$ is regular.
	
	Finally, we show $N$ is $G$-stable. For $g,h,k\in G$ we have
	\begin{align*}
	^k\eta_g[h] &= \lambda(k)\eta_g\lambda(k^{-1})[h]\\
	&= \lambda(k)\eta_g[k^{-1}h]\\
	&=kg\psi(g^{-1})k^{-1}h\psi(g)\\
	&=\big(kg\psi(g^{-1})k^{-1}\psi(g)\big)\psi(g^{-1})h\psi(g),
	\end{align*}
	and since $\psi(kg\psi(g^{-1})k^{-1}\psi(g)) = \psi(g)$ we get
	$^k\eta_g[h] =\eta_{kg\psi(g^{-1})k^{-1}\psi(g)}[h]$
	and $N$ is $G$-stable.
\end{proof}

The opposite subgroup to $N$ is also easy to describe.

\begin{corollary}\label{main2}
	Let $\psi$ be as above. For $g\in G$, define $\eta'_g\in\Perm(G)$ by
	$\eta'_g[h]=h\psi(h^{-1})g\psi(h)$,
	and let $N'=N'_{\psi}=\{\eta'_g:g\in G\}$. Then $N'$ is regular, $G$-stable, and $N'=\mathrm{Cent}_{\Perm(G)}(N_{\psi})$. 
\end{corollary}
\begin{proof}
	This can be established in a manner similar to theorem \ref{main}; alternatively, since $|N|=|N'|$ it suffices to show $N$ and $N'$ commute. Either approach is routine.
\end{proof}

Recall that in the theory of fixed point free abelian endomorphisms it was possible to obtain the same regular, $G$-stable subgroup for two different choices of $\psi$. That remains the case here, but with a simpler criterion.

\begin{proposition}\label{equal}
	Let $\psi_1,\psi_2:G\to G$ be abelian. Then $N_{\psi_1}=N_{\psi_2}$ if and only if there exists a homomorphism $\zeta:G\to Z(G)$ such that $\psi_1(g)=\zeta(g)\psi_2(g)$.
\end{proposition}

\begin{remark}
	This is a clearer notion of equivalence than found in \cite{Childs13}, where $\zeta$ interacts with the $\psi$'s in a more subtle way. The condition is exactly the same: what is different is the presentation of our regular subgroup. 
\end{remark}

\begin{proof}
	For $i=1,2$ write $N_{i}=N_{\psi_i}=\{\eta_{i,g}: g\in G\}$. Suppose that $N_{1}=N_{2}$. Since $\eta_{1,g}[1_G]=\eta_{2,g}[1_G]=g$ we see that $\eta_{1,g}=\eta_{2,g'}$ if and only if $g=g'$. Thus, for all $h\in G$ we have
	\[\eta_{1,g}[h] = g\psi_1(g^{-1})h\psi_1(g) = g\psi_2(g^{-1})h\psi_2(g)=\eta_{2,g}[h],\]
	from which it follows that
	\[\psi_2(g)\psi_1(g^{-1})h(\psi_2(g)\psi_1(g^{-1}))^{-1} = h\]
	so $\psi_2(g)\psi_1(g^{-1})\in Z(G)$. Let $\zeta(g)=\psi_2(g)\psi_1(g^{-1})$. Then $\psi_2(g)=\zeta(g)\psi_1(g)$, and since
	\begin{align*}
	\zeta(gh) &= \psi_2(gh)\psi_1((gh)^{-1})\\
	&= \psi_2(g)(\psi_2(h)\psi_1(h^{-1}))\psi_1(g^{-1})\\
	&=\psi_2(g)\zeta(h)\psi_1(g^{-1})\\
	&= \psi_2(g)\psi_1(g^{-1})\zeta(h)\tag{$\zeta(h)\in Z(G)$}\\
	&= \zeta(g)\zeta(h)
	\end{align*}
	we see that $\zeta:G\to Z(G)$ is the desired homomorphism. The converse--that having such a $\zeta$ shows $N_1=N_2$--is trivial.
\end{proof}

\begin{remark}
	As is evident in the above proof, one does not need to show that $\zeta$ is a homomorphism: $N_{\psi_1}=N_{\psi_2}$ if and only if $\psi_2(g)\psi_1(g^{-1})\in Z(G)$ for all $g\in G$.
\end{remark}

In \cite[Prop. 5.1]{KochTruman20b} we show that if $\varphi\in\Aut(G)$ and $\psi:G\to G$ is a fixed point free abelian endomorphism, then $\varphi^{-1}\psi\varphi$ is also a fixed point free abelian endomorphism. Here, we extend this result to abelian maps, and give a condition for when conjugating by $\varphi^{-1}$ fails to give a new regular subgroup.

\begin{proposition}
	If $\psi:G\to G$ is abelian, and $\varphi\in\Aut(G)$, then $\varphi^{-1}\psi\varphi$ is abelian. Furthermore, $N_{\psi}=N_{\varphi^{-1}\psi\varphi}$ if and only if $\psi(g\varphi(g^{-1}))\in Z(G)$ for all $g\in G$.
\end{proposition}

\begin{proof}
	Let $\psi_{\varphi} = \varphi^{-1}\psi\varphi$. That $\psi_{\varphi}$ is an abelian map is easy to show, and mimics the proof in \cite[Prop. 5.1]{KochTruman20b}.
	For the second statement, by proposition \ref{equal} we know $N_{\psi}=N_{\psi_{\varphi}}$ if and only if $\psi_{\varphi}(g)\psi(g^{-1})\in Z(G)$ for all $g\in G$. We have
	\[
	\psi_{\varphi}(g)\psi(g^{-1}) = (\varphi^{-1}\psi\varphi(g))(\psi(g^{-1}))
	=\varphi^{-1}\psi(\varphi(g)g^{-1}),
\]
which is in $Z(G)$ if and only if $\psi\big( \varphi(g)g^{-1}\big)\in Z(G)$, which is true if and only if its inverse, $\psi\big( g\varphi(g^{-1})\big)$, is in $Z(G)$.
	\end{proof}

%

Let us consider some examples. These are generalizations of examples presented in \cite{KochStordyTruman20}.

\begin{example}\label{Sn}
	Let $n\ge 5$, and suppose $\psi:S_n\to S_n$ is abelian. Since $\ker \psi \triangleleft S_n$ we must have $\ker \psi = \{1_{S_n}\}, A_n$, or $S_n$. As $\psi$ is abelian we know that $\ker \psi \ne \{1_{S_n}\}$, and $\ker \psi = S_n$ if and only if $\psi$ is the trivial map. We shall assume $\ker \psi = A_n$.
	
	Since $S_n$ is generated by transpositions it suffices to describe $\psi(\tau)$ for all transpositions $\tau$. Furthermore, since $\psi$ is abelian and all transpositions are conjugate, $\psi(\tau_1)=\psi(\tau_2)$ for all transpositions $\tau_1,\tau_2\in S_n$. Since $\tau^2\in A_n$ we know $\psi(\tau)$ has order $2$. So let $\xi\in S_n$ have order $2$, and define
	\[\psi(\sigma):=\psi_{\xi}(\sigma) = \begin{cases}
	1_G & \sigma\in A_n \\
	\xi & \sigma\not\in A_n
	\end{cases}.\]
	This is clearly an endomorphism, and since $\psi(S_n)=\gen{\xi}\cong C_2$ it is abelian. We can see that $\psi$ is fixed point free if and only if $\xi\in A_n$. The corresponding regular subgroup is $N=\{\eta_{\sigma}:\sigma\in S_n\}$ with
	\[\eta_{\sigma}[\pi] = \begin{cases}
	\sigma\pi & \sigma \in A_n\\
	\sigma \xi \pi \xi & \sigma\notin A_n
	\end{cases}. \]
	Since $Z(S_n)$ is trivial, each choice of $\xi$ produces a different regular, $G$-stable subgroup. Note that if we extend the choices of $\xi$ to include $\xi=1_G$ we also have the trivial map in this classification.

	We can also compute the elements of the opposite subgroup: $N'=\{\eta_{\sigma}':\sigma\in G\}$ with 
	\[\eta'_{\sigma}[\pi] = \begin{cases}
\pi\sigma & \pi \in A_n\\
\pi \xi \sigma \xi & \pi\notin A_n
\end{cases}. \]

We will see later that the subgroups above capture all of the regular, $S_n$-stable subgroups of $S_n$ in the case $n=5$.
	
\end{example}

\begin{example}\label{meta}
	Let $p>q$ be primes, $p\equiv 1\pmod q$, and let $M_{p,q}$ denote the nonabelian metacyclic group of order $pq$, namely
	\[M_{p,q}=\gen{s,t:s^p=t^q=1_G,\;tst^{-1}=s^d}\]
		where $d$ is an integer whose (multiplicative) order is $q$ mod $p$. If $\psi:M_{p,q}\to M_{p,q}$ is a nontrivial abelian endomorphism then $\ker \psi=\gen{s}$ since the Sylow $p$-subgroup is normal in $M_{p,q}$ and the Sylow $q$-subgroup is not. Thus $\psi(s)=1_{M_{p,q}}$. Write
		\[\psi(t)=s^it^j,\;0\le i \le p-1,\;1\le j\le q-1.\]
		Then $\psi_{i,j}:=\psi$ is an endomorphism with cyclic image, hence abelian. As observed in \cite[6.6]{KochStordyTruman20}, $\psi$ is fixed point free if and only if $j\ne 1$.
		We will see later that all regular, ${M_{p,q}}$-stable subgroups of $M_{p,q}$ come from abelian maps, either directly or through the opposite construction.
\end{example}

\section{Fixed Point Free Abelian Endomorphisms and Beyond}\label{Guru}

Here, we will show how our work includes all of the constructions in \cite{Childs13}, where the abelian maps are all fixed point free. We will also provide a class of examples which suggests that our construction encompasses significantly more Hopf-Galois structures. Recall that an abelian map $\Psi:G\to G$ is fixed point free if $\Psi(g)=g$ implies $g=1_G$. (Note the slight change in notation, reserving $\psi$ for our abelian maps.) 

Recall that a fixed point free abelian map $\Psi:G\to G$ gives rise to a regular, $G$-stable subgroup
$N=\{\lambda(g)\rho(\Psi(g)): g\in G\}$
which is isomorphic to $G$. As a consequence,
\[N[1_G]=\{g\Psi(g^{-1}):g\in G\} = G, \]
so every $k\in G$ can be represented uniquely as $k=g\psi(g^{-1})$ for some $g\in G$.
Define $\psi:G\to G$ by
$\psi(g\Psi(g^{-1}))=\Psi(g^{-1})$.
We first claim that $\psi$ is a fixed point free abelian endomorphism. (Indeed, it is the quasi-inverse of $\Psi$ as described in \cite{Childs13} and \cite{Caranti13}.) Note that

\[
\big(g\Psi(g^{-1})h\Psi(g)\big)\Psi\big(g\Psi(g^{-1})h\Psi(g)\big)^{-1}=\big(g\Psi(g^{-1})h\Psi(g)\big)\Psi(h^{-1}g^{-1})
=g\Psi(g^{-1})h\Psi(h^{-1}),
\]
so
\begin{align*}
	\psi(g\Psi(g^{-1})h\Psi(h^{-1}))&=\psi\Big(\big(g\Psi(g^{-1})h\Psi(g)\big)\Psi\big(g\Psi(g^{-1})h\Psi(g)\big)^{-1}\Big)\\
	&=\Psi\big(g\Psi(g^{-1})h\Psi(g)\big)^{-1}\\
	&=\Psi(g^{-1}h^{-1})\\
	&=\Psi(g^{-1})\Psi(h^{-1})\\
	&=\psi(g\Psi(g^{-1}))\psi(h\Psi(h^{-1}))
\end{align*}
and $\psi: G\to G$ is a homomorphism. 

Next, if $\psi(g\Psi(g^{-1}))=g\Psi(g^{-1})$ then $g=1_G$ by the definition of $\psi$, hence $\psi$ is fixed point free. Additionally, $\psi(G)=\Psi(G)$ so $\psi$ is abelian. 

Next we claim that, for all $g\in G$, $\eta_{g\Psi(g^{-1})} = \lambda(g)\rho(\Psi(g))$. Indeed, we have
\begin{align*}
	\eta_{g\Psi(g^{-1})}[h]&=(g\Psi(g^{-1}))\psi((g\Psi(g^{-1}))^{-1})h\psi(g\Psi(g^{-1})) \\
	&=g\Psi(g^{-1})\Psi(g)h\Psi(g^{-1})\\
	&=\lambda(g)\rho(\Psi(g))[h].
\end{align*}
Thus, the construction presented in this work includes all of the structures found in \cite{Childs13}.

Conversely, if $\psi:G\to G$ is fixed point free abelian, giving the regular, $G$-stable subgroup $N=\{\eta_g\}$ as above, defining $\Psi:G \to G$ by 
$\Psi(g\psi(g^{-1}))=\psi(g^{-1})$
gives a fixed point free abelian map, and
\[
\lambda(g\psi(g^{-1}))\rho(\Psi(g\psi(g^{-1})))[h] = g\psi(g^{-1})h\Psi\big(\psi(g)g^{-1}\big)\\
=g\psi(g^{-1})h\psi(g)=\eta_g[h].
\]
Thus, there is a one-to-one correspondence between the fixed point free constructions presented here and the constructions in \cite{Childs13}.

The following ``normal complement'' example gives a completely different family of Hopf-Galois structures obtained through abelian maps.

\begin{example}\label{nc}
	Let $G',G''\le G$ with $G'\triangleleft G$, $G''$ abelian, $G'\cap G'' = \{1_G\}$, and $|G'||G''|=|G|$. Define $\psi:G\to G$ by $\psi(hk)=k,\;h\in G',\;k\in G''$. As $G'$ is normal in $G$, $\psi$ is a homomorphism, evidently abelian. The corresponding regular, $G$-stable subgroup of $\Perm(G)$ can be made quite explicit: $N=\{\eta_{hk}:h\in G',k\in G''\}$ with $\eta_{hk}=\lambda(h)\rho(k)$. Thus, $N\cong G'\times G''$.
\end{example}

One obtains from this a proof of \cite[Cor. 6]{CrespoRioVela16} (see also \cite[Cor. 7.2]{Childs19}) in the case where the (potentially) non-normal group (here, $G''$) is abelian.

Note that if $G'$ is nonabelian then we have the opposite subgroup $N'=\{\eta'_{hk}:h\in G',k\in G''\}$ with
$\eta'_{hk}[xy] = xy\psi((xy)^{-1})hk\psi(xy) = xhky^{-1},\;x\in G',\;y\in G''$ as can be readily computed.

Returning to example \ref{Sn}, where an abelian map $\psi:S_n\to S_n,\;n\ge 5$ depends on the choice of a $\xi\in S_n$ with $\xi^2=1_{S_n}$, we observed that $\psi$ was fixed point free if and only if $\xi\in A_n$. Now, observe that if $\xi\notin A_n$ then $S_n=A_n\gen{\xi}$ and the above applies, producing a Hopf-Galois structure of type $A_n\times C_2$. In \cite[Th. 5, Th. 9]{CarnahanChilds99} there is a complete description of regular, $G$-stable subgroups of $\Perm(S_n),\;n\ge 5$ of type $S_n$ as well as of type $A_n\times C_2$. The classification given here, together with the opposite groups, account for all such subgroups. Furthermore, \cite[Prop. 5]{CrespoRioVela18} states that the only regular, $S_5$-stable subgroups of $\Perm(S_5)$ are of type $S_5$ or $A_5\times C_2$. Thus abelian maps give us all desired subgroups when $n=5$.

Similarly, in example \ref{meta} an abelian map $\psi:M_{p,q}\to M_{p,q}$ depends on integers $0\le i\le p-1,\;1\le j \le q-1$; furthermore $\psi$ is fixed point free if and only if $j\ne 1$. The case $j=1$ can be interpreted using example \ref{nc} by writing $M_{p,q}=\gen{s}\gen{s^it}$, giving a Hopf-Galois structure of type $C_p\times C_q$. The Hopf-Galois structures on a metacyclic extension are fully described in \cite{Byott04}. They are all of type $M_{p,q}$ or $C_p\times C_q$. Using the characterization found in \cite[\S 8]{KochTruman20b} it is clear that we have found all such structures here (once opposites are considered in the type $M_{p,q}$ case).

\section{Abelian Maps and Braces}

In \cite[Prop. 4.4]{KochStordyTruman20} the brace corresponding to a fixed point free abelian endomorphism is found. Here we will duplicate this result while allowing our abelian map to have fixed points. 

In fact, our description of the regular, $G$-stable subgroup $N:=N_{\psi}$ (for $\psi:G\to G$ abelian) allows for a simpler proof than the one given in \cite{KochStordyTruman20}. The primary reason for this is that the map $\varkappa: N\to G$ is particularly nice using our construction:
$\varkappa(\eta_g) = \eta_{g}[1_G] = g$.
Then, since
$\eta_g\eta_h[1_G] = \eta_g[h] = g\psi(g^{-1})h\psi(g)$
we see that $\eta_g\eta_h = \eta_{g\psi(g^{-1})h\psi(h)}$ and $\eta_g^{-1} = \eta_{\psi(g)g^{-1}\psi(g^{-1})}$.
 If we define
$
	\eta_g\circ\eta_h = \varkappa^{-1}(\varkappa(g)\varkappa(h))
	=\varkappa^{-1}(gh)
	=\eta_{gh}
$
then $(N,\cdot,\circ)$ is a brace. 

It would seem desirable to think of our underlying set as $G$ instead of $N$. If we identify, through $\varkappa$, the elements of $N$ with the elements of $G$, then the dot operation becomes 
$g\cdot h = \varkappa(\eta_g\eta_h) = \varkappa(\eta_{g\psi(g^{-1})h\psi(h)}) = g\psi(g^{-1})h\psi(g)$
and the circle operation becomes
$g\circ h = \varkappa(\eta_g\circ \eta_h) = \varkappa(\eta_{gh}) = gh$.
This creates an inconvenient issue with notation, as we can no longer suppress the dot expressions without creating confusion. However, this can be remedied: we claim that $(G,\circ,\cdot)$ is also a brace, which we can prove by showing the brace relation holds on $(G,\circ,\cdot)$. Swapping our operations this way, for $g,h,k\in G$ we need to show 
$g\cdot(h\circ k) = (g\cdot h)\circ \overline{g} \circ (g\cdot k)$.
We have
\begin{align*}
(g\cdot h)\circ \overline{g} \circ (g\cdot k) &= (g\psi(g^{-1})h\psi(g))\circ \overline g \circ (g\psi(g^{-1})k\psi(g))\\
&=(g\psi(g^{-1})h\psi(g))g^{-1}(g\psi(g^{-1})k\psi(g))\\
&=g\psi(g^{-1})hk\psi(g)
=g\circ (hk).
\end{align*}

The brace constructed above is what Childs, in \cite{Childs19} calls a {\it bi-skew brace}. In short, a bi-skew brace is any (skew left) brace where interchanging the operations results in another (skew left) brace. Noticing that our brace is bi-skew allows us to choose which is the dot operation. For reasons mentioned above, it seems reasonable to exchange the operations, giving the following.

\begin{proposition}
	Let $\psi$ be an abelian map on $(G,\cdot)$. Then $\B_{\psi}=(G,\cdot,\circ)$ is a bi-skew brace, with $g\cdot h = gh$ and
	$g\circ h = g\psi(g^{-1})h\psi(g)$. 
	Furthermore, $(G,\circ)\cong N_{\psi}$.
\end{proposition}

Of course, this can also be verified by simply checking $(G,\circ)$ is a group and the brace relation holds. The brace constructed is identical to the brace in \cite{KochStordyTruman20} in the fixed point free case.


The observation that $\B_{\psi}$ is bi-skew allows us to construct another Hopf-Galois structure on a Galois extension with, potentially, a different Galois group. In order to minimize confusion below we will adopt very explicit notation for all of our binary operations.

\begin{theorem}
	Let $\psi:(G,\cdot)\to (G,\cdot)$ be an abelian endomorphism, and define $(G,\circ)$ as above. Suppose $(N,\ast_N)$ is an abstract group which is isomorphic to $(G,\circ)$. Then $\psi$ gives rise to a regular, $N$-stable subgroup $P\le\Perm(N)$ with $P\cong G$.
\end{theorem}

\begin{proof}
	Let $\alpha: (G,\circ)\to (N,\ast_N)$ be an isomorphism. For $g\in G$, define $\pi_g: N \to N$ by
$\pi_g[n] = \alpha(g\cdot\alpha^{-1}(n))$,
and let $P=\{\pi_g : g\in G\}$. Then
\begin{align*}\pi_g \pi_h[n] &= \pi_g[\alpha(h\cdot \alpha^{-1}(n))]\\
&= \alpha(g\cdot \alpha^{-1}(\alpha(h)\ast_N n))\tag{$\alpha^{-1}:N\to G$ is a homomorphism}\\
&=\alpha(g\cdot h\cdot \alpha^{-1}(n))
=\pi_{g\cdot h}[n],
\end{align*}
so $P\le \Perm(N)$ and $P\cong G$. One can easily show $P$ is regular.

Furthermore, we have
\begin{align*}
^m\pi_g[n] &= \lambda(m)\pi_g\lambda(m^{-1})[n]\\
&= m\ast_N\alpha(g\cdot\alpha^{-1}(m^{-1}\ast_N n))\\
&= m\ast_N \alpha(\alpha^{-1}(\alpha(g))\cdot\alpha^{-1}(m^{-1}\ast_N n))\\
&=m\alpha(\alpha^{-1}(\alpha(g)\ast_N m^{-1}\ast_N n))\\
&= m\ast_N\alpha(g)\ast_Nm^{-1}\ast_Nn,
\end{align*}
whereas
\begin{align*}
\pi_{\alpha^{-1}(m)\cdot g\cdot \alpha^{-1}(m^{-1})}[n] &= \alpha(\alpha^{-1}(m)\cdot g\cdot\alpha^{-1}(m^{-1})\cdot\alpha^{-1}(n))\\
&=\alpha(\alpha^{-1}(m)\cdot \alpha^{-1}(\alpha(g))\cdot\alpha^{-1}(m^{-1})\cdot\alpha^{-1}(n))\\
&=m\ast_N\alpha(g)\ast_Nm^{-1}\ast_Nn,
\end{align*}
hence $^m\pi_g=\pi_{\alpha^{-1}(m)\cdot g\cdot \alpha^{-1}(m^{-1})}$ and $P$ is $N$-stable. 
\end{proof}


\begin{example}
Let $G=S_n$, $\xi\in S_n$ an odd permutation of order $2$, and suppose $n\ge 5$. Let $\psi=\psi_{\xi}$ as in example \ref{Sn}. Then we have seen that $N_{\psi}\cong A_n\times C_2$, hence there is be a Hopf-Galois structure on an $A_n\times C_2$ extension of type $S_n$. An isomorphism $\alpha:(G,\circ)\to A_n\times C_2$ corresponding to our choice of $\xi$ is given by
$\alpha(\sigma)=(\sigma,1_{C_2}),\;\alpha(\tau) = (\tau\xi, \xi),\;\sigma\in A_n,\;\tau\notin A_n,\;C_2=\gen{\xi}$.
Then
\[\pi_{\tau}[(\sigma,1_{C_2})] = \begin{cases}
(\tau\sigma,1_{C_2}) & \tau\in A_n\\
(\tau\sigma\xi,\xi) & \tau\notin A_n
\end{cases},\;
\pi_{\tau}[(\sigma,\xi)] = \begin{cases}
(\tau\sigma,1_{C_2}) & \tau\in A_n\\
(\tau\sigma,\xi) & \tau\notin A_n
\end{cases}. \]
\end{example}


Of course, the construction above depends on a choice of isomorphism $\alpha: (G,\circ)\to N$. If $\beta:(G,\circ)\to N$ is another isomorphism, then $\alpha\beta^{-1}\in\Aut(N)$ and the resulting regular subgroups differ by conjugation by $\alpha\beta^{-1}$. 

We have described how one can use a brace $(B,\cdot,\circ)$ to construct two set-theoretic solutions to the Yang-Baxter equation. When a brace is in fact bi-skew, we get (up to) four solutions, namely
\begin{align*}
R_{\B}(x,y) &= (x^{-1}(x\circ y),\overline{x^{-1}(x\circ y)}\circ x \circ y)\\
R_{\B'}(x,y) &= ((x\circ y)x^{-1},\overline{(x\circ y)x^{-1}}\circ x \circ y)\\
S_{\B}(x,y) &= (\overline x\circ (xy),(\overline x\circ (xy))^{-1}xy)\\
S_{\B'}(x,y) &= ( (xy)\circ \overline x,yx( (xy)\circ \overline x)^{-1}).\\
\end{align*}
Applying these formulas to our bi-skew brace gives the following.
\begin{corollary}
	Let $\psi:G \to G$ be an abelian endomorphism. Then each of the following is a set-theoretic solution to the Yang-Baxter equation:
	\begin{align*}
	R_{1,\psi}(g,h) &= \big(\psi(g^{-1})h\psi(g),\psi(hg^{-1})h^{-1}\psi(g)g\psi(g^{-1})h\psi(gh^{-1})\big)\\
	R_{2,\psi}(g,h) &= \big(g\psi(g^{-1})h\psi(g)g^{-1},\psi(h)g\psi(h^{-1})\big)\\
	R_{3,\psi}(g,h) &= \big(\psi(g)h\psi(g^{-1}),\psi(g)h^{-1}\psi(g^{-1})gh\big)\\
	R_{4,\psi}(g,h) &= \big( gh\psi(h^{-1})g^{-1}\psi(h),hg\psi(h^{-1})g\psi(h)h^{-1}g^{-1}\big),\;g,h\in G.\\
	\end{align*}
	Furthermore:
	\begin{enumerate}
		\item $R_{1,\psi}=R_{2,\psi}$ if and only if $G$ is an abelian group (in which case $R_{1,\psi}(g,h)=R_{2,\psi} = (h,g)$).
		\item $R_{3.\psi}=R_{4,\psi}$ if and only if $g\psi(g^{-1})h\psi(g) = h\psi(h^{-1})g\psi(h)$ for all $g,h\in G$.
		\item $R_{1,\psi}R_{2,\psi}=R_{3,\psi}R_{4,\psi}=\id$. 
	\end{enumerate} 
\end{corollary}

\begin{proof}
	The solutions are straightforward computations, and (1)--(3) follow from properties of opposite braces.
\end{proof}

\section{Five subgroups of \texorpdfstring{$G$}{G} and \texorpdfstring{$N$}{N}, and five sub-Hopf-algebras of \texorpdfstring{$K[N]^G$}{K[N]\^{}G}}

Recall that if $\psi:G \to G$ is a fixed point free abelian map, then we obtain a Hopf Galois structure of type $G$ whose Hopf algebra $L[N]^G$ is isomorphic to $H_{\lambda}$ as $K$-Hopf algebras. Once we allow $\psi$ to have fixed points our structure may no longer be of type $G$, nor need the Hopf algebra be isomorphic to $H_{\lambda}$. We ask: can we determine the type of the Hopf-Galois structure arising from an abelian map $\psi$? In general, there appears to be no easy way to determine the isomorphism type of $N_{\psi}$, however we are able to obtain some results about this group's structure.

We will investigate questions on Hopf-Galois structure through the use of five subgroups of $G$ that depend on $\psi$. Each of these subgroups give rise to a (non-regular) $G$-stable subgroup of $\Perm(G)$, which in turn will give a sub-Hopf algebra of $H=L[N]^G$.

In \cite{KochTruman20b}, the concepts of $\lambda$-points and $\rho$-points were introduced  to investigate questions involving brace equivalence. Given a regular, $G$-stable subgroup $N\le \Perm(G)$ the sets of $\lambda$-points and $\rho$-points, denoted $\Lambda_N$ and $P_N$ respectively, are defined as
\[\Lambda_N=N\cap\lambda(G),\;P_N= N \cap \rho(G).\]
Both $\Lambda_N$ and $P_N$ are subgroups of $N$.
Note that if $N_1$ and $N_2$ are brace equivalent then $\Lambda_{N_1}\cong\Lambda_{N_2}$ and $P_{N_1}\cong P_{N_2}$ \cite[Prop. 6.3]{KochTruman20b}.

First, we let $G_0=\ker \psi,\;N_0=\{\eta_{g_0}:g_0\in G_0\}$. Since
$\eta_{g_0}[h]=g_0\psi(g_0^{-1})h\psi(g_0) = g_0h,\;g_0\in G_0, \;h\in G$ we have $N_0=\lambda(G_0)\le N$. In particular, note that $N_0\cong G_0$, providing us some information as to the structure of $N$. Clearly, $G_0\triangleleft G$, and
$\eta_g\eta_{g_0}\eta_g^{-1}=\eta_g\eta_{{g_0}g^{-1}}=\eta_{g\psi(g^{-1})g_0g^{-1}\psi(g)}$, and since
$\psi(g\psi(g^{-1})g_0g^{-1}\psi(g))=\psi(g_0)=1_G$ we see $g\psi(g^{-1})g_0g^{-1}\psi(g)\in\ker\psi$, hence $N_0\triangleleft N$ as well.

 Since $N_0=\lambda(G_0)$ is $G$-stable, by \cite[Th. 5.2]{GreitherPareigis87} (see also  \cite[Prop. 2.2]{KochKohlTrumanUnderwood19a} for a more explicit formulation), $H_0:=(L[N_0])^G=L[\lambda(G_0)]^G$ is a sub-Hopf algebra of $H$ which is also contained in $H_{\lambda}$.

Generalizing slightly, let $\widehat G_0=\psi^{-1}(Z(G))$. This is evidently a subgroup of $G$ containing $G_0$, and let
$\widehat N_0:=\{\eta_{\hat g_0}:\hat g_0\in \widehat G_0\}$.

\begin{lemma}
	With notation as above, $\widehat N_0 = \Lambda_N$.
\end{lemma}
\begin{proof}
	For $\hat g_0\in \widehat G_0, \;h\in G$ we have
	$\eta_{\hat g_0}[h]=\eta_{\hat g_0 \psi(\hat g_0^{-1}) h \psi(\hat g_0)}=\hat g_0 h$
	so clearly $\widehat N_0 \subset \Lambda_N$. Conversely, if $\eta_g\in \Lambda_N$ then $\eta_g=\lambda(k)$ for some $k\in G$: evaluating this expression at $1_G$ shows $g=k$. Thus,
	$\eta_g[h]=g\psi(g^{-1})h\psi(g) = gh$,
	so $\psi(g^{-1})h\psi(g) = h$ for all $h\in G$. This can only occur if $\psi(g)\in Z(G)$.
\end{proof}

As both $N$ and $\lambda(G)$ are $G$-stable, so is $\widehat N_0$. Thus, $\widehat H_0 := (L[\widehat N_0])^G$ is a sub-Hopf algebra of $H$. It is precisely the sub-Hopf algebra of $H_{\lambda}$ obtained by restricting to $\widehat G_0$, i.e., $\widehat H_0 = L[\lambda(\widehat G_0)]^G$. 

Next, let $G_1=\{g_1\in G:\psi(g_1)=g_1\}$ be the subgroup of fixed points. Clearly $G_1$ is abelian, although typically $G_1$ is not normal in $G$. If we let $N_1=\{\eta_{g_1}:g_1\in G_1\}$ then $\eta_{g_1}[h]=g_1g_1^{-1}hg_1 = hg_1$
for all $h\in G$, hence $\eta_{g_1}=\rho(g_1^{-1})$. Thus $N_1 = \rho(G_1)$ is a subgroup of $N$ isomorphic to $G_1$.

We also have
$^k\eta_{g_1} = \eta_{kg_1\psi(g_1^{-1})k^{-1}\psi(g_1)}=\eta_{\psi(g_1)}$,
so $G$ acts trivially on $N_1$. Thus, $H_1=(L[N_1])^G=K[G_1]$ is a sub-Hopf algebra of $H$. 

Generalizing again, let $\phi:G\to G$ be given by $\phi(g)=g\psi(g^{-1})$ for all $g\in G$. Define $\widehat G_1 = \{\hat g\in G: \phi(\hat g)\in Z(G)\}$.  As $\phi$ is trivial on fixed points we clearly have $G_1 \le \widehat{G}_1$. Since 
$\phi(\hat g \hat h)= \hat g\hat h \psi(\hat h^{-1} \hat g^{-1}) = \hat g \phi(\hat h) \psi(\hat g^{-1}) = \phi(\hat g)\phi(\hat h)$
for $\hat h\in \widehat G_1$ we see this is in fact a subgroup of $G$. If we define
$\widehat N_1 = \{\eta_{\hat g_1}:\hat g_1\in \widehat G_1\}$
we get
\begin{lemma}
	With notation as above, $\widehat N_1 = P_N$.
\end{lemma}
\begin{proof}
	For $\hat g_1\in \widehat G_1$ we have
	$\eta_{\hat g_1}[h] = {\hat g_1 \psi(\hat g_1^{-1}) h \psi( \hat g_1)} = {\phi(\hat g_1) h \psi (\hat g_1)} = {h \phi(\hat g_1)\psi(\hat g_1)}=h\hat g_1$
	hence $\eta_{\hat g_1}=\rho(\hat g_1^{-1})\in P_N$. Conversely, suppose $\eta_{ g}=\rho(k)$. By evaluating at $1_G$ we see $k=g^{-1}$, hence for all $h\in G$ we have
	\begin{align*}\eta_g[h]=g\psi(g^{-1})h\psi(g) &= hg\\ 
	g\psi(g^{-1}) h &= hg\psi(g^{-1})\\
	\phi(g) h &=h \phi(g),
	\end{align*}
	thus $\phi(g)\in Z(G)$, i.e., $g\in\widehat G_1$. Therefore, $\eta_g\in \widehat{N}_1$.
\end{proof}

Since $^k\eta_{\hat g_1} = \eta_{k\phi(\hat g_1)k^{-1}\psi(\hat g_1)}= \eta_{\hat g_1}$ we see that $G$ acts trivially on $\widehat N_1$. Then $\widehat H_1:=(L[\widehat{N}])^G = K[\widehat G_1]$ is a sub-Hopf algebra of $H$. 
In fact, $\widehat H_1$ is the largest group ring contained in $H$.

Finally, note that $G_0\cap G_1$ is trivial, hence $N_0\cap N_1 = \{1_N\}$. Since $G_0\triangleleft G$ and $N_0\triangleleft N$ we get \[G_{01}:=G_0G_1=\{g_0g_1: g_0\in G_0,\;g_1\in G_1\}\le G, \; N_{01}:=N_0N_1=\{\eta_{g_0g_1}: g_0\in G_0,\;g_1\in G_1\}\le N.\]
We have
$\eta_{g_0g_1}[h]=g_0g_1\psi(g_1^{-1}g_0^{-1})h\psi(g_0g_1)=g_0hg_1$
hence $\eta_{g_0g_1}=\lambda(g_0)\rho(g_1^{-1})$. The group $N_{01}$ is evidently $G$-stable, giving rise to another sub-Hopf algebra $H_{01}:=(L[N_{01}])^G \cong (L[N_0])^G \otimes K[G_1]$. 

The construction of the subgroups above allow us to obtain the following. Note the relationship between the following result and example \ref{nc}.

\begin{proposition}\label{xprod}
	Let $\psi:G\to G$ be abelian, and let $G_0,G_1,N_0,N_1$ be as above. Then $N_{\psi}$ has a subgroup isomorphic to $G_0\times G_1$. In particular, if $|G_0||G_1|=|G|$ then $N_{\psi}\cong G_0\times G_1$. 
\end{proposition}

\begin{proof}
	The work above creates the subgroup $N_{0,1}=N_0N_1$. That $N_{0,1}\cong N_0\times N_1$ follows from
	$\eta_{g_0g_1}\eta_{h_0h_1} = \lambda(g_0)\rho(g_1^{-1})\lambda(h_0)\rho(h_1^{-1}) = \lambda(g_0h_0)\rho(g_1h_1) = \eta_{g_0h_0g_1h_1}$.
\end{proof}

We conclude with an investigation of dihedral groups.
Let $G=D_n=\gen{r,s:r^n=s^2=rsrs=1_G}.$ We will find all abelian maps on $G$ and determine the type of each Hopf-Galois structure. Suppose $\psi:G \to G$ is an abelian map. Since
		$\psi(r)\psi(s)=\psi(rs)=\psi(sr)=\psi(r^{-1}s)=\psi(r^{-1})\psi(s)$ we know that $\psi(r)$ must be an element whose order divides both $2$ and $n$. 
	We will examine two cases, based on the parity of $n$.
	
	Suppose first that $n$ is odd. Then $\psi(r)=1_G$. Letting $\psi(s)=r^is$ for some $0\le i \le n-1$ gives an abelian map, and it is clear that every nontrivial abelian map is of this form. Since $Z(D_n)=\{1_G\}$ each choice of $i$ gives a different regular, $G$-stable subgroup. For each, $G_0=\gen{r}$ and $G_1=\gen{r^is}$, so the resulting Hopf-Galois structure is of type $C_n\times C_2$. Thus we have $n+1$ Hopf-Galois structures, one of type $D_n$ and $n$ of type $C_n\times C_2$.
	
	Now suppose $n$ is even. Then $\psi(r)=1_G, r^{n/2}$, or $r^is$ for some $0\le i \le n-1$; and $\psi(s)=1_G, r^{n/2}$, or $r^js$ for some $0\le j \le n-1$. However, since $r^{n/2}\in Z(D_n)$ we need only study the following cases:
	\begin{description}
		\item[Case 1. $\psi(r)=\psi(s)=1_G$]Then $\psi=\mathrm{id}$, and the Hopf-Galois structure is of type $D_n$.
		\item[Case 2. $\psi(r)=1_G,\psi(s)=r^js,\;0\le j\le (n/2)-1$] As in the case $n$ is odd we get $G_0=\gen{r}$ and $G_1=\gen{r^js}$, hence we get $n/2$ Hopf-Galois structures of type $C_n\times C_2$.
		\item[Case 3. $\psi(r)=r^is,\psi(s)=1_G,\;0\le i \le (n/2)-1$] Since $\psi(r^is) = (r^is)^i$ we see that $r^is$ is a fixed point if an only if $i$ is odd. If $i$ is even then $\psi$ is fixed point free, hence the corresponding Hopf-Galois structure is of type $D_n$. On the other hand, if $i$ is odd, then $G_0=\gen{r^2,s}\cong D_{n/2}$ (note $D_2\cong C_2\times C_2$) and $G_1=\gen{r^is}$, hence the Hopf-Galois structure is of type $D_{n/2}\times C_2$. 
		Overall, this case gives $n/4$ Hopf-Galois structures of type $D_n$ and $n/4$ Hopf-Galois structures of type $D_{n/2}\times C_2$ if $n\equiv 0 \pmod 4$; and $n/2$ Hopf-Galois structures of type $D_n$ if $n\equiv 2\pmod 4$ (since $D_{n/2}\times C_2\cong D_n$).
		\item[Case 4. $\psi(r)=\psi(s)=r^is,\;0\le i\le (n/2)-1$] (Note that $\psi(r)=r^is,\;\psi(s)=r^js,\;0\le i,j\le (n/2)-1$ is abelian if and only if $i=j$.) In this case, $r^is$ is a fixed point if and only if $i$ is even. Thus, for $i$ odd we get a Hopf-Galois structure of type $D_n$. On the other hand, if $i$ is even then $G_0=\gen{r^2,rs}\cong D_{n/2},\;G_1=\gen{r^is}$, and the Hopf-Galois structure is of type $D_{n/2}\times C_2$.
	\end{description}
		The Hopf-Galois structure types are summarized in the following table. The rightmost column counts the total number of regular, $G$-stable subgroups found using abelian maps directly or through their opposite structures.
		
		\[\begin{array}{c | c | c | c | c | c}
		n & \text{\# type } D_n & \text{\# type }D_{n/2}\times C_2 & \text{\# type }C_{n}\times C_2 & \text{ \# abelian maps} & \text{\# HGS}\\\hline
		2\mid n  & 1+n/2  & n/2 & n/2 &  1+3n/2 & 2+5n/2\\\hline
		2\nmid n & 1 & \text{N/A} & n& 1+n & 2+n\\
		\end{array}
		 \]
		 
		 The number of structures of type $D_n$ (or $D_{n/2}\times C_2$ if $n\equiv 2\pmod 4$) arising from fixed point free abelian maps agrees with the results found in \cite[\S 5]{Childs13} before opposites are considered.

\bibliographystyle{alpha} 
\bibliography{../../MyRefs}
\end{document}